\documentclass[11pt]{amsart}
\usepackage[babel]{csquotes}
\usepackage{enumitem}
\usepackage{amsmath,amsthm,amssymb,mathrsfs,amsfonts,verbatim,enumitem,color,leftidx,mathtools,bm}
\usepackage{mathabx}
\usepackage{etoolbox} 
\usepackage{bbm}
\usepackage[all,tips]{xy}
\usepackage{graphicx,ifpdf}
\usepackage{stmaryrd}
\usepackage{amssymb}
\usepackage{dsfont}

\ifpdf
\fi
\usepackage[colorlinks]{hyperref}
\hypersetup{
linkcolor=blue,          
citecolor=green,        
}

\usepackage{tikz}
\usetikzlibrary{calc}
\usetikzlibrary{shapes.misc}

\tikzset{cross/.style={cross out, draw=black, minimum size=2*(#1-\pgflinewidth), inner sep=0pt, outer sep=0pt},
cross/.default={1pt}}

\newtheorem{thm}{Theorem}[section]

\newtheorem{cor}[thm]{Corollary}
\newtheorem{prop}[thm]{Proposition}

\theoremstyle{definition}
\newtheorem{defi}[thm]{Definition}

\theoremstyle{remark}
\newtheorem{rem}[thm]{Remark}

\numberwithin{equation}{section}
\definecolor{esperance}{rgb}{0.0,0.5,0.0}

\DeclareMathOperator{\diam}{diam}

\newcommand{\Ga}{\Gamma}
\newcommand{\del}{\delta}
\newcommand{\Del}{\Delta}
\newcommand{\lam}{\lambda}
\newcommand{\Lam}{\Lambda}
\newcommand{\eps}{\epsilon}

\newcommand{\cD}{\mathcal{D}}

\newcommand{\cP}{\mathcal{P}}

\newcommand{\bR}{\mathbb{R}}
\newcommand{\bZ}{\mathbb{Z}}

\newcommand{\bN}{\mathbb{N}}

\newcommand{\bT}{\mathbb{T}}

\newcommand{\SL}{\operatorname{SL}}

\newcommand\ul[1]{\underline{#1}}

\newcommand\on[1]{\operatorname{#1}}

\newcommand\mb[1]{\mathbf{#1}}

\newcommand\tb[1]{\textbf{#1}}

\newcommand{\Supp}{\on{Supp}}

\newcommand{\onto}{\xymatrix{\ar@{>>}[r]&}}

\newcommand{\eq}[1]
{
\begin{equation*}
{#1}
\end{equation*}
}
\newcommand{\eqlabel}[2]
{
\begin{equation}
{#2}\label{#1}
\end{equation}
}

\makeatletter
\newcommand*{\rom}[1]{\expandafter\@slowromancap\romannumeral #1@}
\makeatother

\begin{document}

\title[High entropy measures with escape of mass]{High entropy measures on the space of lattices with escape of mass}

\date{}

\author{Taehyeong Kim}
\address{Department of Mathematics, Brandeis University, Waltham MA}
\email{taehyeongkim@brandeis.edu}

\thanks{The author was supported by the ERC grant HomDyn, ID 833423.}



\def\thefootnote{}
\def\thefootnote{\arabic{footnote}}
\setcounter{footnote}{0}

\begin{abstract}
For any diagonal element $a$ with two eigenvalues, we construct a sequence of $a$-invariant probability measures on the space of unimodular lattices with high entropy but converging to the zero measure. This extends the result of Kadyrov [Ergodic Theory Dynam. Systems, \textbf{32}(1) (2012)].
\end{abstract}
\maketitle
\section{Introduction}
Let $m,n$ be positive integers. Consider the homogeneous space $X_{m+n}=\SL_{m+n}(\bR)/\SL_{m+n}(\bZ)$, which can be identified with the space of unimodular lattices in $\bR^{m+n}$. We consider the diagonal flow
\[
a_t = \begin{pmatrix} e^{t/m}I_m & \\ & e^{-t/n}I_n \end{pmatrix} \quad\text{ for } t\in\bR,
\]
which naturally acts on $X_{m+n}$ by left multiplication. We let $a=a_1$ be the time-one map for the diagonal flow.

The following non-escape of mass result was proved in \cite{KKLM17} using \textit{Margulis functions} which originated in \cite{EMM98}.
\begin{thm}\cite{KKLM17}\label{Thm_KKLM}
For any $h>0$ and any sequence $(\mu_k)_{k\geq1}$ of $a$-invariant probability measures on $X_{m+n}$ with $h_{\mu_k}(a)\geq h$, any weak$^\ast$ limit $\mu$ of the sequence satisfies
\[
\mu(X_{m+n})\geq h-(m+n-1).
\]
\end{thm}

A similar statement was first proved in \cite{ELMV12} for the $\SL_2(\bR)/\SL_2(\bZ)$ case. Since then, generalizations were considered in \cite{EKP15, Mor22} for other rank $1$ cases, as well as in \cite{EK12, Mor} for higher rank cases. See also \cite{DKMS} for ``almost" invariant measures.

Note that the maximal entropy for the transformation $a$ on $X_{m+n}$ is $m+n$ (see e.g. \cite{EL,KKLM17}) and Theorem \ref{Thm_KKLM} is nontrivial for $h\in(m+n-1,m+n]$. Kadyrov, Kleinbock, Lindenstrauss, and Margulis \cite{KKLM17} conjectured that Theorem \ref{Thm_KKLM} is sharp in the sense that for any $h \in [m+n-1, m+n]$ there should exist a sequence of probability invariant measures $(\mu_k)_{k\geq 1}$ with $\lim_{k\to\infty} h_{\mu_k}(a)=h$ such that the limit measure $\mu$ satisfies $\mu(X_{m+n})=h-(m+n-1)$. This was proved in \cite{Kad12} when $\min(m,n)=1$. See also \cite{KP17} for a result of this type for arbitrary rank-one Lie groups, and \cite{RV19} for the geodesic flow on negatively curved geometrically finite manifolds.

In this article, we prove that the conjecture is true when $(m,n)\neq (1,1)$, hence combining with \cite{Kad12} it is true for every $m$ and $n$. 
\begin{thm} \label{Thm_main}
There exists a sequence of $a$-invariant probability measures $(\mu_k)_{k\geq 1}$ on $X_{m+n}$ with $\lim_{k\to\infty}h_{\mu_k}(a) = m+n-1$ such that its weak$^\ast$ limit is the zero measure.
\end{thm}

\begin{cor} \label{Cor_main}
For any $h\in [m+n-1,m+n]$, there exists a sequence of $a$-invariant probability measures $(\nu_k)_{k\geq 1}$ on $X_{m+n}$ with $\lim_{k\to\infty}h_{\nu_k}(a)=h$ such that its weak$^\ast$ limit $\nu$ satisfies $\nu(X_{m+n})=h-(m+n-1)$.
\end{cor}

\textit{Structure of the paper.} 
In Section \ref{Sec2}, using the variational principle developed in \cite{DFSU}, we estimate the lower bound of the Hausdorff dimension for the set of matrices such that the diagonal orbit of the corresponding lattice stays in a certain compact region in $X_{m+n}$ near the cusp.
In Section \ref{Sec3}, we construct an $a$-invariant probability measure on $X_{m+n}$ with high entropy supported on that compact region near the cusp.
In Section \ref{Sec4}, we prove Theorem \ref{Thm_main} and Corollary \ref{Cor_main}.

\vspace{5mm}
\tb{Acknowledgments}. 
I am grateful to Jinho Jeoung for introducing this problem to me.
I also thank David Simmons for answering my question about the standard template, and the anonymous referees for their careful reading of this article and helpful comments.

\section{Dimension estimates using variational principle}\label{Sec2}
In this section, we will estimate the lower bound of the Hausdorff dimension for the set of matrices such that the diagonal orbit of the corresponding lattice stays in a certain compact region in $X_{m+n}$ near the cusp. We make use of the variational principle in the parametric geometry of numbers developed in \cite{DFSU}. 

For given positive integers $m$ and $n$, denote $d=m+n$.
For each $j=1,\dots,d$, let $\lam_j(\Lam)$ denote the $j$th successive minimum of a lattice $\Lam$ in $\bR^d$, i.e. the infimum of $\lam$ such that the set $\{\mb{r}\in\Lam:\|\mb{r}\|\leq \lam\}$ contains $j$ linearly independent vectors.
Given a matrix $A\in M_{m,n}(\bR)$, we define the \textit{successive minima function} $\mb{h}=\mb{h}_A=(h_1,\dots,h_d):[0,\infty)\to\bR^d$ of the matrix $A$ by
\[
h_i(t)=\log\lam_i (a_t u_A \bZ^d), \quad\text{where} \quad u_A = \begin{pmatrix}I_m & A \\ & I_n \end{pmatrix}.
\]
We use the following notation:
\[
[a,b]_\bZ = [a,b]\cap \bZ \quad \text{and}\quad (a,b]_\bZ = (a,b]\cap \bZ.
\]
Now we recall the definition of templates in \cite{DFSU}.
\begin{defi}[Template] An $m\times n$ \textit{template}\label{Def_Temp} is a piecewise linear map $\mb{f}:[0,\infty)\to\bR^d$ with the following properties:
\begin{enumerate}
\item $f_1\leq \cdots\leq f_d.$
\item\label{Temp_item_2} $-\frac{1}{n}\leq f_i'\leq \frac{1}{m}$ for all $i$.
\item For all $j=0,\dots,d$ and for every interval $I$ such that $f_j<f_{j+1}$ on $I$, the function
$\sum_{0<i\leq j}f_i$ is convex and piecewise linear on $I$ with slopes in the set
\[
\left\{\frac{L_{+}}{m}-\frac{L_{-}}{n}: L_{+}\in [0,m]_\bZ, L_{-}\in[0,n]_\bZ, L_{+}+L_{-}=j\right\}.
\] As a convention, we use $f_0 =-\infty$ and $f_{d+1}=+\infty$.
\end{enumerate}
\end{defi}

\begin{defi}[Contraction rate of a template] 
Let $\mb{f}$ be a template and $I$ be an open interval on which $\mb{f}$ is linear. 
An \textit{interval of equality} for $\mb{f}$ on $I$ is an interval $(p,q]_\bZ$ with $0\leq p<q\leq d$ such that
\[
f_p<f_{p+1}=\cdots=f_q < f_{q+1}\ \text{on}\ I.
\]
For $0\leq p<q\leq d$ with an interval of equality $(p,q]_\bZ$, we let $M_{\pm}(p,q)$ be the unique integers such that
\[
M_{+}+M_{-}=q-p\quad\text{and}\quad \sum_{i=p+1}^q f_i' = \frac{M_{+}}{m}-\frac{M_{-}}{n}\text{ on } I.
\]
Note that $M_{\pm}\geq 0$ by \eqref{Temp_item_2} of Definition \ref{Def_Temp} (see page 24 in \cite{DFSU}). We also let
\[
S_{+}  = \bigcup_{(p,q]_\bZ} (p,p+M_{+}(p,q)]_\bZ \quad\text{and}\quad
S_{-} = \bigcup_{(p,q]_\bZ} (p+M_{+}(p,q), q]_\bZ
\]
where the unions are taken over all intervals of equality for $\mb{f}$ on $I$. Note that $S_{+} \cap S_{-} =\varnothing$, $S_{+}\cup S_{-}=[1,d]_\bZ$, $\#S_{+}=m$, and $\#S_{-}=n$.
Define
\[\begin{split}
\del(\mb{f},I) &= \#\{(i_{+},i_{-})\in S_{+}\times S_{-}: i_{+}<i_{-}\};\\
\ul{\del}(\mb{f})&= \liminf_{T\to \infty} \frac{1}{T}\int_0^T \del(\mb{f},t) dt,
\end{split}\]
where $\del(\mb{f},t)$ is the piecewise constant function with value $\del(\mb{f},I)$ on $I$.
The value $\ul{\del}(\mb{f})$ is called the \textit{lower average contraction rate} of $\mb{f}$.
\end{defi}

The following \cite[Theorem 4.7]{DFSU} is a uniform version of the variational principle.
\begin{thm}[Uniform variational principle]\label{Thm_VP} 
For all $\eps>0$, there exists $C>0$ such that for every template $\mb{f}$,
\[
\dim_H(\cD(\mb{f},C)) \geq \ul{\del}(\mb{f})-\eps,
\]
where 
\[
\cD(\mb{f},C) = \{A\in M_{m,n}(\bR): \|\mb{h}_A-\mb{f}\|\leq C \}.
\]
\end{thm}

In \cite[Section 9]{DFSU}, the following special template was mainly used to estimate the lower bound of the Hausdorff dimension of singular matrices. We also need that template for our purpose.
\begin{defi}[Standard template]
For $k\in \bZ_{\geq 0}$, fix $0\leq t_k <t_{k+1}$ and $\eps_k,\eps_{k+1}\geq 0$ and let $\Del t=\Del t_k = t_{k+1}-t_k$ and $\Del \eps=\Del \eps_k = \eps_{k+1}-\eps_k$. Assume that the following formulas hold:
\begin{align}
-\frac{1}{m}\Del t \leq \Del \eps &\leq \frac{1}{n}\Del t, \label{Eq_ST1} \\ 
\Del\eps \geq -\frac{n-1}{2n}\Del t \text{ if } m=1 \text{ and } &\Del\eps \leq \frac{m-1}{2m}\Del t \text{ if } n=1, \label{Eq_ST2}\\
\text{either } (n-1)(\frac{1}{n}\Del t -\Del \eps) \geq d\eps_k \text{ or } &(m-1)(\frac{1}{m}\Del t +\Del\eps) \geq d\eps_{k+1}. \label{Eq_ST3}
\end{align}
The \textit{standard template} defined by the two points $(t_k, -\eps_k)$ and $(t_{k+1},\eps_{k+1})$ is the partial template $\mb{f}:[t_k,t_{k+1}]\to\bR^d$ defined as follows:
\begin{itemize}
\item Let $g_1,g_2:[t_{k},t_{k+1}]\to\bR$ be piecewise linear functions such that $g_i(t_j) = -\eps_j$, and $g_i$ has two intervals of linearity: one on which $g_i ' =\frac{1}{m}$ and another on which $g_i' =-\frac{1}{n}$. For $i=1$ the latter interval comes first while for $i=2$ the former interval comes first. Finally, let $g_3=\cdots=g_d$ be chosen so that $g_1+\cdots+g_d =0$.
\item For each $t\in[t_k,t_{k+1}]$ let $\mb{f}(t)=\mb{g}(t)$ if $g_2 (t) \leq g_3 (t)$; otherwise let $f_1(t)=g_1(t)$ and let $f_2(t)=\cdots =f_d(t)$ be chosen so that $f_1+\cdots+f_d = 0$. 
\end{itemize}
Denote the standard template defined by $(t_k,-\eps_k)$ and $(t_{k+1},-\eps_{k+1})$ by $\mb{s}[(t_k,-\eps_k),(t_{k+1},-\eps_{k+1})]$.
\end{defi}

\begin{rem}
As explained in \cite[Definition 9.1 and Lemma 9.2]{DFSU}, the formulas \eqref{Eq_ST1}, \eqref{Eq_ST2}, and \eqref{Eq_ST3} are necessary to ensure the existence of the standard template.
In particular, we will choose $t_k = kt$ and $\eps_k = C>0$ with some large $t >0$ and some constant $C>0$ for all $k\geq 1$, hence $\Del\eps_k =0$. See the proof of Proposition \ref{Prop_dim}. Note that if $\Del t_k = t$ is large enough compared to the constant $C$, the formulas \eqref{Eq_ST1}, \eqref{Eq_ST2}, and \eqref{Eq_ST3} hold unless $(m,n)= (1,1)$. If $(m,n)=(1,1)$, it follows from \eqref{Eq_ST2} and \eqref{Eq_ST3} that $\eps_k = \eps_{k+1} =0$. In order to choose $\eps_k = C>0$, we should assume that $(m,n)\neq (1,1)$.
\end{rem}

\begin{prop}\label{Prop_dim} Suppose that $(m,n) \neq (1,1)$.
For any $\eps>0$, there are constants $\rho_{\eps},\eta_{\eps}>0$ and $t_\eps\in \bN$ such that
$$\dim_H \{A\in M_{m,n}(\bR) : \rho_\eps \leq \lam_1(a_t u_A\bZ^d) \leq \eta_\eps \text{ for all }t\geq t_\eps\} \geq mn-\frac{mn}{m+n}-\eps.$$
Moreover, $\rho_\eps, \eta_{\eps}$ converge to $0$ as $\eps \to 0$.
\end{prop}
\begin{proof}
We will use the following notation: given any template $\mb{g}$ and interval $[T_1,T_2]$, we denote $$\Del(\mb{g},[T_1,T_2]) = \frac{1}{T_2 - T_1} \int_{T_1}^{T_2} \del(\mb{g},t)dt.$$

Fix $\eps>0$ and let $C_\eps>0$ be as in the statement of Theorem \ref{Thm_VP}. We may assume that $C_\eps \to \infty$ as $\eps \to 0$ by replacing $C_\eps$ with $\max(C_\eps, 1/\eps)$.
Fix large $t>0$ to be determined later. Define the template $\mb{f}$ by
\[\mb{f} = \begin{dcases}
\mb{s}[(0,0),(t,-2C_\eps)] & \text{ on } [0,t]\\
\mb{s}[(kt,-2C_\eps),((k+1)t,-2C_\eps)] & \text{ on } [kt,(k+1)t] \text{ for all } k\geq 1.
\end{dcases}\]
Note that if $t>0$ is large enough, then the formulas \eqref{Eq_ST1}, \eqref{Eq_ST2}, and \eqref{Eq_ST3} hold for all the above standard templates since $(m,n) \neq (1,1)$.

Following \cite[Section 9]{DFSU}, we have that for each $k\geq 1$
\[\begin{split}
\Del(\mb{f},[kt,(k+1)t])&=\Del(\mb{s}[(kt,-2C_\eps),((k+1)t,-2C_\eps)],[kt,(k+1)t])\\
&=\Del\left(\mb{s}\left[\left(0,-\frac{2C_\eps}{t}\right),\left(1,-\frac{2C_\eps}{t}\right)\right],[0,1]\right)\\
&=\Del(\mb{s}[(0,0),(1,0)],[0,1])-O\left(\frac{C_\eps}{t}\right)\\
&=mn-\frac{mn}{m+n}-O\left(\frac{C_\eps}{t}\right).
\end{split}\]
By Theorem \ref{Thm_VP}, $\dim_H(\cD(\mb{f},C_\eps)) \geq mn-\frac{mn}{m+n}-\eps-O\left(\frac{C_\eps}{t}\right)$.
Take $t=t_\eps\in\bN$ large so that $\dim_H(\cD(\mb{f},C_\eps)) \geq mn-\frac{mn}{m+n}-2\eps$.

Fix any $A \in \cD(\mb{f},C_\eps)$ and let $\mb{h}_A = (h_1,\dots,h_d)$. Observe that $$h_1(t) = f_1(t) - (f_1(t)-h_1(t)) \leq -2C_\eps + C_\eps = -C_\eps$$ for all $t\geq t_\eps$. Note that $\mb{f}$ is a bounded template. Setting $D_\eps = \max_{t} (-f_1(t))$, $$h_1(t) = f_1(t) - (f_1(t)-h_1(t))\geq -D_\eps-C_\eps$$ for all $t\geq t_\eps$. 
Hence, it follows that
\[\begin{split}
\dim_H\{A\in M_{m,n}(\bR) :  -D_\eps -C_\eps \leq \log\lam_1(a_t u_A \bZ^d) &\leq -C_\eps, \forall t\geq t_\eps \}\\
 &\geq mn-\frac{mn}{m+n}-2\eps.
\end{split}\]
Since $-D_\eps -C_\eps$ and $-C_\eps \to -\infty$ as $\eps \to 0$, taking $\rho_\eps =  e^{-D_{\eps/2} -C_{\eps/2}}$ and $\eta_\eps = e^{-C_{\eps/2}}$, we can conclude Proposition \ref{Prop_dim}.
\end{proof}

\section{Construction of high entropy measures}\label{Sec3}
In this section, we will construct an $a$-invariant probability measure on $X_{m+n}$ with high entropy supported on a certain compact region near the cusp. We basically follow the strategy as in \cite{LSS, KKL,KLP23} to construct invariant measures and make use of Proposition \ref{Prop_dim} to bound the entropy.  

Let $d=m+n$ as in Section \ref{Sec2}. Denote $G=\SL_{d}(\bR)$, $\Ga=\SL_{d}(\bZ)$, and $X=G/\Ga$. Let $d_G$ be a right invariant metric on $G$ and let $d_X$ be the metric on $X$ induced by $d_G$. 
Denote by $d_\infty$ the metric on $G$ induced by the supremum norm on $M_{d,d}(\bR)$.
Since $d_G$ and $d_\infty$ are locally bi-Lipschitz (see e.g. \cite[Lemma 9.12]{EW}), there are constants $0<r_0<1$ and $C_0\geq 1$ such that if $d_G(g,id)<r_0$ or $d_\infty(g,id)<r_0$, then
\eqlabel{Eq_biLip}{
\frac{1}{C_0}d_\infty (g,id) \leq d_G(g,id)\leq C_0 d_\infty(g,id).
}

We refer the reader to \cite[Chapters 1 and 2]{ELW} for definitions and properties of entropies. In particular, for a countable partition $\cP$ of $X$ and a probability measure $\mu$ on $X$, the entropy $H_{\mu}(\cP)$ is defined by
\[
H_\mu (\cP) = -\sum_{P\in\cP} \mu(P) \log \mu(P) \in [0,\infty]
\] where $0\log 0$ is defined to be $0$. The (dynamical) entropy $h_{\mu}(a)$ of the transformation $a$ on $X$ is defined by
\[
h_\mu (a) = \sup_{\cP: H_{\mu}(\cP)<\infty} \lim_{N\to\infty} \frac{1}{N} H_{\mu}\left(\bigvee_{k=0}^{N-1}a^{-k}\cP\right),
\]
where $\bigvee_{k=0}^{N-1}a^{-k}\cP$ denotes the join of the preimages $a^{-k}\cP$.

\begin{prop}\label{Prop_construction}
Suppose that $(m,n)\neq (1,1)$. Fix $\eps>0$ and let $\rho_\eps,\eta_\eps$ be as in Proposition \ref{Prop_dim}. Then there exists an $a$-invariant probability measure $\mu$ on $X$ such that 
\begin{enumerate}
\item $\Supp\mu \subset  \{x\in X : \rho_\eps \leq \lam_1(x) \leq \eta_\eps \}$;
\item $h_\mu(a) \geq m+n-1-\frac{m+n}{mn}\eps$.
\end{enumerate}
\end{prop}
\begin{rem}\label{rem_compact}
The set $\{x\in X : \rho_\eps \leq \lam_1(x) \leq \eta_\eps \}$ is a compact subset of $X$ by Mahler's compactness criterion.
\end{rem}
\begin{proof}[Proof of Proposition \ref{Prop_construction}]
Given $\eps>0$, let $\rho_\eps,\eta_\eps>0$ and $t_\eps\in\bN$ be as in Proposition \ref{Prop_dim}. Denote 
\[
E(\eps)=\{A\in \bT^{mn}: \eta_{\eps}\leq \lam_1(a_t u_A \bZ^{d}) \leq \rho_\eps \text{ for all }t\geq t_\eps\},
\] where $\bT^{mn}=\bR^{mn}/\bZ^{mn}$ with the metric $\|\cdot\|$ on $\bT^{mn}$ induced by the supremum norm $\|\cdot\|$ on $\bR^{mn}=M_{m,n}(\bR)$. 
For each $N\in\bN$, let $S_N$ be a maximal $e^{-(\frac{1}{m}+\frac{1}{n})N}$-separated subset of $E(\eps)$. 
Since the value $\lam_1(a_t u_A \bZ^d)$ is invariant under $\bZ^{mn}$-translations, i.e. $\lam_1(a_t u_{A+Z} \bZ^d)=\lam_1(a_t u_A \bZ^d)$ for any $Z\in\bZ^{mn}$, it follows from
Proposition \ref{Prop_dim} that
\eqlabel{Eq_BoxDim}{
\liminf_{N\to\infty} \frac{\log |S_N|}{(\frac{1}{m}+\frac{1}{n})N} \geq \dim_H E(\eps) \geq mn-\frac{mn}{m+n}-\eps,
}
Define the measures $\nu_N$ and $\mu_N$ by 
\[
\nu_N =\frac{1}{|S_N|} \sum_{A\in S_N} \del_{u_A\bZ^d} \quad\text{and}\quad 
\mu_N=\frac{1}{N}\sum_{k=0}^{N-1} a^k_{\ast}\nu_N.
\]
Extracting a subsequences if necessary, we may assume that the sequence $(\mu_N)_{N\in\bN}$ converges weak$^\ast$ to some measure $\mu$ on $X$. Note that $\mu$ could be not a probability measure since $X$ is non-compact.

First, observe that $\mu$ is clearly $a$-invariant since $a_\ast \mu_N - \mu_N$ goes to the zero measure. Next, we claim that $\mu$ is a probability measure on $X$ and 
$$\Supp\mu \subset  \{x\in X : \rho_\eps \leq \lam_1(x) \leq \eta_\eps \}.$$
Indeed, writing $X(\eps) = \{x\in X : \rho_\eps \leq \lam_1(x) \leq \eta_\eps \}$,
it follows from the definition of $E(\eps)$ that
\[
\mu_N(X\smallsetminus X(\eps)) = \frac{1}{N|S_N|} \sum_{A\in S_N}\sum_{k=0}^{N-1} \del_{a^k u_A \bZ^{d}} (X\smallsetminus X(\eps))
\leq \frac{t_\eps}{N}.
\]
By taking $N\to\infty$, we have $\mu(X\smallsetminus X(\eps)) = 0$ and the claim follows since $X(\eps)$ is compact in $X$.

Finally, we claim that
\eq{
h_\mu (a) \geq m+n-1- \frac{m+n}{mn}\eps.
}
In order to prove the claim, we fix small $r>0$ which will be determined 
and consider a finite partition $\cP$ of $X$ satisfying
\begin{itemize}
\item $\cP$ contains an atom $P_\infty$ such that $X\smallsetminus P_\infty$ contains $X(\eps)$ and has compact closure;
\item For any $P\in \cP\smallsetminus \{P_\infty\}$, $\diam P<r$;
\item For any $P \in \cP$, the boundary of $P$ is $\mu$-null, i.e. $\mu(\partial P) =0$.
\end{itemize}
We can build such $\cP$ following the procedure in the proof of \cite[Proposition 2.3]{LSS}.
For any $q\in\bN$, let $\cP^{(q)} = \bigvee_{k=0}^{q-1}a^{-k}\cP$. Write the Euclidean division of large enough $N-1$ by $q$ as
$$N-1=qN' +s \text{ with } s\in \{0,\dots,q-1\}. $$ 
By subadditivity of entropy of the partition it follows that for each $p\in\{0,\dots,q-1\}$ 
\[
H_{\nu_N}(\cP^{(N)}) \leq H_{a^{p}\nu_N}(\cP^{(q)}) +H_{a^{p+q}\nu_N}(\cP^{(q)})+\cdots+H_{a^{p+qN'}\nu_N}(\cP^{(q)})+2q\log|\cP|.
\]
Summing those inequalities for $p=0,\dots,q-1$, and using concavity of entropy of the measure, it follows that
\[
q H_{\nu_N}(\cP^{(N)}) \leq \sum_{k=0}^{N-1}H_{a^k \nu_N} (\cP^{(q)}) + 2q^2 \log|\cP| \leq N H_{\mu_N}(\cP^{(q)}) +2q^2\log|\cP|.
\]
Therefore, we have
\eqlabel{Eq_entropybound}{
\frac{1}{q} H_{\mu_{N}} (\cP^{(q)}) \geq \frac{1}{N}H_{\nu_N}(\cP^{(N)}) -\frac{2q\log|\cP|}{N}.
}

Now we will take $r >0$ small enough. For this denote by $r_1$ the injectivity radius of $X\smallsetminus P_\infty$. See \cite[Proposition 9.14]{EW} for the definition and existence of an injectivity radius. 
Recall the constants $0<r_0<1$ and $C_0\geq 1$ given in the beginning of Section \ref{Sec3}. 
We take $r < \min\{r_0, r_1,C_0^{-1}e^{-\left(\frac{1}{m}+\frac{1}{n}\right)}\}$ and claim that for any non-empty atom $P$ of $\cP^{(N)}$, there is at most one element $A\in S_N$ such that $u_A\Ga \in P$.

To prove the claim, suppose that $A_1,A_2\in S_N$ satisfy $u_{A_1}\Ga, u_{A_2}\Ga \in P$, that is, $a^ku_{A_1}\Ga, a^k u_{A_2}\Ga$ are contained in the same atom of $\cP$ for each $k=0,\dots,N-1$. 
It follows from $A_1,A_2 \in E(\eps)$ that $a^k u_{A_1}\Ga, a^k u_{A_2}\Ga \notin P_\infty$ for all $k\geq t_\eps$. In particular, we have
$d_X(a^{N-1} u_{A_1} \Ga, a^{N-1} u_{A_2} \Ga) <r$, which implies $d_G(a^{N-1}u_{A_1-A_2}a^{-(N-1)},id) <r$ due to $r<r_1$.
It follows from $r<r_0$ that
\[
d_\infty(a^{N-1}u_{A_1-A_2}a^{-(N-1)},id)=e^{\left(\frac{1}{m}+\frac{1}{n}\right)(N-1)}\|A_1-A_2\| < C_0 r,
\]
hence $\|A_1-A_2\| < C_0 r e^{-\left(\frac{1}{m}+\frac{1}{n}\right)(N-1)}$. Since $ r< C_0^{-1}e^{-\left(\frac{1}{m}+\frac{1}{n}\right)}$ and $A_1,A_2$ are $e^{-\left(\frac{1}{m}+\frac{1}{n}\right)N}$-separated, it follows that $A_1=A_2$, which completes the proof of the claim.

It follows from \eqref{Eq_entropybound} and the above claim that 
\[
\frac{1}{q}H_{\mu_N}(\cP^{(q)}) \geq \frac{\log |S_N|}{N} -\frac{2q\log|\cP|}{N}.
\]
Since the boundary of the atoms of $\cP$, hence of $\cP^{(q)}$, is of zero $\mu$-measure, taking $N\to\infty$ and using \eqref{Eq_BoxDim}, we have
\[
\frac{1}{q}H_{\mu}(\cP^{(q)}) \geq m+n-1-\frac{m+n}{mn}\eps.
\] Therefore, we have
\[
h_\mu(a) \geq m+n-1-\frac{m+n}{mn}\eps.
\]
\end{proof}

\section{Proofs of Theorem \ref{Thm_main} and Corollary \ref{Cor_main}}\label{Sec4}
\begin{proof}[Proof of Theorem \ref{Thm_main}]
When $(m,n)=(1,1)$, it follows from \cite[Theorem 1.2]{Kad12}. Assume $(m,n)\neq (1,1)$. Using Proposition \ref{Prop_construction} with $\eps=1/k$ for $k\in\bN$, we obtain a sequence of $a$-invariant probability measures $(\mu_k)_{k\geq 1}$ on $X$ such that $\Supp \mu_k \subset \{x\in X : \rho_{1/k} \leq \lam_1(x) \leq \eta_{1/k}\}$ and $h_{\mu_k}(a)\geq m+n-1 -\frac{m+n}{mnk}$. Hence, it follows that $\liminf_{k\to\infty}h_{\mu_k}(a) \geq m+n -1$ and any weak$^\ast$ limit measure is the zero measure since $\rho_{1/k}, \eta_{1/k}\to 0$ as $k\to\infty$. 

On the other hand, if $\limsup_{k\to\infty}h_{\mu_k}(a)>m+n-1$, it follows from Theorem \ref{Thm_KKLM} that there is a subsequence of $(\mu_k)_{k\geq 1}$ such that any weak$^\ast$ limit measure along this subsequence cannot be the zero measure, which is a contradiction. Therefore, $\lim_{k\to\infty}h_{\mu_k}(a) = m+n -1$.
\end{proof}

\begin{proof}[Proof of Corollary \ref{Cor_main}]
Let $(\mu_k)_{k\geq 1}$ be as in Theorem \ref{Thm_main}. Denote by $m_X$ the Haar probability measure on $X$. 
Note that $h_{m_X} (a) = m+n$ (see e.g. \cite[Theorem 7.9]{EL}). Given $h\in [m+n-1,m+n]$, 
define $$\nu_k = (h-(m+n-1))m_X + (m+n-h)\mu_k.$$ 
It follows from \cite[Theorem 2.33]{ELW} that $$h_{\nu_k}(a)= (h-(m+n-1))h_{m_X}(a) + (m+n-h)h_{\mu_k}(a),$$ hence $\lim_{k\to\infty}h_{\nu_k}(a) =h$. Since the weak$^\ast$ limit measure of $(\nu_k)_{k\geq 1}$ is $(h-(m+n-1))m_X$, this completes the proof.
\end{proof}


\begin{thebibliography}{KKLM17}

\bibitem[DFSU24]{DFSU}
T. Das, L. Fishman, D. Simmons and M. Urba\'{n}ski, \emph{A variational principle in the parametric
geometry of numbers}, Advances in Mathematics \textbf{437} (2024): 109435.

\bibitem[DKMS25]{DKMS}
O. David, T. Kim, R. Mor, and U. Shapira \emph{On the rate of convergence of continued fraction statistics of random rationals}, Selecta Mathematica, \textbf{31}(2) (2025), 1--22.

\bibitem[EK12]{EK12}
M. Einsiedler and S. Kadyrov, \emph{Entropy and escape of mass for
$\SL(3, \bZ)\ \SL(3, \bR)$}, Israel J. Math. \textbf{190} (2012), 253–288.

\bibitem[EKP15]{EKP15}
M. Einsiedler, S. Kadyrov, and A. D. Pohl, \emph{Escape of mass and entropy for diagonal flows
in real rank one situations}, Israel J. Math. \textbf{210} (2015), 245–295.

\bibitem[EL10]{EL}
M. Einsiedler and E. Lindenstrauss, \emph{Diagonal actions on locally homogeneous spaces}, Homogeneous flows, moduli spaces and arithmetic, Clay Math. Proc., vol.~10, Amer. Math. Soc., Providence, RI, 2010, pp.~155--241.
  \MR{2648695}

\bibitem[ELMV12]{ELMV12}
M. Einsiedler, E. Lindenstrauss, P. Michel and A. Venkatesh, \emph{The distribution
of periodic torus orbits on homogeneous spaces, II: Duke’s theorem for quadratic
fields}, Enseign. Math. \textbf{58} (2012), 249–313.

\bibitem[ELW]{ELW}
M. Einsiedler, E. Lindenstrauss, and T. Ward, \emph{Entropy in Ergodic Theory and Topological Dynamics}, in preparation, preliminary version available at \href{https://tbward0.wixsite.com/books/entropy}{https://tbward0.wixsite.com/books/entropy}.

\bibitem[EMM98]{EMM98}
A. Eskin, G. A. Margulis and S. Mozes, \emph{Upper bounds and asymptotics in a quantitative
version of the Oppenheim conjecture}, Ann. of Math. (2), \textbf{147} (1998), no. 1, 93-141.

\bibitem[EW11]{EW}
M. Einsiedler and T. Ward, \emph{Ergodic theory with a view towards number theory}, Graduate Texts in Mathematics, vol. 259, Springer-Verlag London, Ltd., London, 2011.

\bibitem[Kad12]{Kad12}
S. Kadyrov, \emph{Positive entropy invariant measures on the space of lattices with escape of mass}, Ergodic Theory and Dynamical Systems \textbf{32}(1) (2012), 141--157.

\bibitem[KKL25]{KKL}
T. Kim, W. Kim, and S. Lim, \emph{Dimension estimates for badly approximable affine forms}, Ergodic Theory and Dynamical Systems \textbf{45}(5) (2025), 1541--1596.

\bibitem[KKLM17]{KKLM17}
S. Kadyrov, D. Kleinbock, E. Lindenstrauss, and G. A. Margulis, \emph{Singular systems of linear forms and
non-escape of mass in the space of lattices}, J. Anal. Math. \textbf{133} (2017), 253–-277.

\bibitem[KLP23]{KLP23}
T. Kim, S. Lim, and F. Paulin, \emph{On Hausdorff dimension in inhomogeneous Diophantine approximation over global function fields}, Journal of Number Theory, \textbf{251} (2023), 102–146.

\bibitem[KP17]{KP17}
 S. Kadyrov and A. D. Pohl, \emph{Amount of failure of upper-semicontinuity of entropy in noncompact rank one situations, and Hausdorff dimension}, Ergodic Theory and Dynamical Systems \textbf{37} (2017), 539--563.


\bibitem[LSS19]{LSS}
S. Lim, N. de Saxc\'e, and U. Shapira, \emph{Dimension bound for badly approximable grids}, 
Int. Math. Res. Not. IMRN, (20) (2019), 6317–6346.

\bibitem[Mor22]{Mor22}
R. Mor, \emph{Excursions to the cusps for geometrically finite hyperbolic orbifolds
and equidistribution of closed geodesics in regular covers}, Ergodic Theory Dynam. Systems \textbf{42} (2022), no. 12, 3745–3791. 

\bibitem[Mor25]{Mor}
R. Mor, \emph{Bounding entropy for one-parameter diagonal flows on
$\SL_d(\bR)/\SL_d(\bZ)$ using linear functionals}, J. Eur. Math. Soc. (2025), published online first.

\bibitem[RV19]{RV19}
F. Riquelme and A. Velozo, \emph{Escape of mass and entropy for geodesic flows},  Ergodic Theory and Dynamical Systems \textbf{39} (2019), no.2, 446--473.



\end{thebibliography}

\def\cprime{$'$} \def\cprime{$'$} \def\cprime{$'$}
\providecommand{\bysame}{\leavevmode\hbox to3em{\hrulefill}\thinspace}
\providecommand{\MR}{\relax\ifhmode\unskip\space\fi MR }

\end{document}